\documentclass[a4paper]{article}
\usepackage{amssymb}
\usepackage{amsmath}
\usepackage{amsthm}
\usepackage{hyperref}
\usepackage{enumerate}
\usepackage{mathtools}
\usepackage{todonotes}
\usepackage{xcolor,url}
\usepackage[utf8]{inputenc}
\usepackage[T2A]{fontenc}

\newtheorem{Theorem} {Theorem} [section]
\newtheorem{Proposition} [Theorem] {Proposition}
\newtheorem{Lemma} [Theorem] {Lemma}

\newtheorem{Example} [Theorem] {Example}

\newtheorem{Conjecture} [Theorem] {Conjecture}

\newcommand{\Ff}{{\mathbb F}}

\newcommand{\cA}{{\mathcal A}}
\newcommand{\cB}{{\mathcal B}}
\newcommand{\cC}{{\mathcal C}}

\newcommand{\cF}{{\mathcal F}}
\newcommand{\cG}{{\mathcal G}}

\newcommand{\cS}{{\mathcal S}}

\renewcommand{\phi}{\varphi} 
\newcommand{\gauss}[2]{\genfrac{[}{]}{0pt}{}{#1}{#2}}

\title{The Erd\H{o}s-Rado Sunflower Problem for Vector Spaces}
\author{\renewcommand\thefootnote{\alph{footnote}}
Ferdinand Ihringer\footnotemark[1] \and
\renewcommand\thefootnote{\alph{footnote}}
Andrey Kupavskii\footnotemark[2]
}

\date{18 September 2025}

\begin{document}
\maketitle

{\renewcommand\thefootnote{\alph{footnote}}
\footnotetext[1]{Dept.~of Mathematics,
Southern University of Science and Technology, Shenzhen, Guangdong, China.
E-mail: {\tt ihringer@sustech.edu.cn}}

\footnotetext[2]{Moscow Institute of Physics and Technology, Dolgoprudniy, Russia.
E-mail: {\tt kupavskii@ya.ru}}}

\begin{abstract}
  The famous Erd\H{o}s-Rado sunflower conjecture suggests that
  an $s$-sun\-flower-free family of $k$-element sets has size at most $(Cs)^k$
  for some absolute constant $C$.
  In this note, we investigate the analog problem for $k$-spaces over the field with $q$ elements.
  For $s \geq k+1$, we show that the largest $s$-sunflower-free family $\cF$ satisfies
  \[
     1 \leq |\cF| / q^{(s-1) \binom{k+1}{2} - k} \leq (q/(q-1))^k.
  \]
  For $s \leq k$, we show that
  \[
     q^{-\binom{k+1}{2}} \leq |\cF| / q^{(s-1) \binom{k+1}{2} - k} \leq (q/(q-1))^k.
  \]
  Our lower bounds rely on an iterative construction that uses lifted maximum rank-distance (MRD) codes.
\end{abstract}

\section{Introduction}

A family of $s$ sets $S_1, \ldots, S_s$ is called an \emph{$s$-sunflower} or a \emph{$\Delta(s)$-system}
with \emph{kernel $K$} if $K = S_i \cap S_j$ for any distinct $i,j.$
\footnote{%
The term ``$\Delta$-system'' goes back to Erd\H{o}s and Rado.
The term ``sunflower'' appeared in a paper by Deza and Frankl \cite{DF1981} and since late 80s was  used by the Boolean Circuit Complexity community, until it eventually took over the $\Delta$-system terminology.}
\footnote{%
Here a \emph{family} is a set, not a multiset.
}
In 1960, Erd\H{o}s and Rado proved
that for $s \geq 3$ an $s$-sunflower-free family of $k$-sets has size
at most $k! (s-1)^k$ \cite{ER1960}.
Erd\H{o}s and Rado suggested that the following is not unlikely:

\begin{Conjecture}[Erd\H{o}s-Rado Sunflower Conjecture]
 Let $s \geq 3$. Then there exists a constant $C$
 such that any $s$-sunflower-free family of $k$-sets $\cF$ satisfies $|\cF| \leq (Cs)^k$.
\end{Conjecture}

One can obtain an $s$-sunflower-free family of $k$-sets of size $(s-1)^k$ considering a complete $k$-partite hypergraph with parts of size $s-1.$
Recently, a breakthrough made by Alweiss, Lovett, Wu, and Zhang \cite{ALWZ2021} allowed to push the bound much closer to the conjectured one. A slight refinement of their bound was obtained by Bell, Chueluecha
and Warnke \cite{BCW}, we know the following.

\begin{Theorem}[Alweiss-Lovett-Wu-Zhang \cite{ALWZ2021}, Bell-Chueluecha-Warnke \cite{BCW}]
 Let $s \geq 3$.
 Any $s$-sunflower-free family of $k$-sets $\cF$ satisfies $|\cF| \leq (C s \log k)^k$ for some absolute $C$.
\end{Theorem}

Here we consider the natural analog of this problem in finite vector spaces.
For a positive integer $n$ and a prime power $q$,
let $V = V(n, q)$ denote the finite vector space of dimension $n$ over the field
with $q$ elements. We say that a family of subspaces $\cS$ is in \emph{general position}
if $\dim( \sum_{S \in \cS} S) = \sum_{S \in \cS} \dim(S)$.
A family of $s$ $k$-spaces $S_1, \ldots, S_s$ of $V$ is called an \emph{$s$-sunflower with kernel $K$}
if there exists a $d$-space $K$ such that $K = S_i \cap S_j$ for any distinct $i,j$
and $S_1/K, \ldots, S_s/K$ are in general position.
In other words, $\dim(S_1 + \cdots + S_s) = d+s(k-d)$.

The aim of this note is to prove the following theorem. See the next section for the definition of Gaussian coefficients $[n]_q$.

\begin{Theorem}\label{thm:main}
 Let $s \geq 3$, $k \geq 2$, and $q$ a prime power. Let $\cF$ be an $s$-sunflower-free family $\cF$ of $k$-spaces over the finite field with $q$ elements. Then $$|\cF|\le \prod_{i=1}^k[i(s-1)]_q \le \Big(\frac q{q-1}\Big)^k q^{(s-1) \binom{k+1}{2} - k}.$$
 Moreover, there exist such families $\cF$ that satisfy the following bounds:
 \begin{enumerate}
  \item For $s \geq k+1$,
 \[
    |\cF| \ge q^{(s-1) \binom{k+1}{2} - k}.
 \]
 \item For $s \leq k$,
 \[
   |\cF|\geq q^{(s-2) \binom{k+1}{2} + k \frac{\mathrm{gcd}(k, s-1) - 3}{2}} \geq q^{(s-2) \binom{k+1}{2} - k}.
 \]
 \end{enumerate}
\end{Theorem}

The upper bound is a straightforward adaptation of the
 $k! (s-1)^k$ bound by Erd\H{o}s and Rado. The main content of the theorem is in the lower bound, which relies on a recursive
nesting of lifted maximum-rank distance (MRD) codes which go back to Gabidulin \cite{Gabidulin1985}. One can see that in the case of fixed $k$ and large $s$ the size of the largest $s$-sunflower-free $\cF$ is determined up to a constant.

There is an alternative definition of a sunflower in the literature on subspace codes
which replaces the condition that the subspaces $S_1/K, \ldots, S_s/K$ are in general position by the condition that they are just pairwise disjoint (pairwise in general position), see Etzion and Raviv \cite{ER2015}. Here the literature is focused on
generalizing a classical result by Deza \cite{Deza1973}, for instance, see \cite{BDBD2022}.
If we replace the condition that
$S_1/K, \ldots, S_d/K$ are just \emph{pairwise disjoint}, then we speak of
an \emph{$s$-set-like sunflower with kernel $K$}.
Bounds and construction for $s$-set-like sunflowers can be obtained in a similar way as in the set case. We omit their discussion for that reason. We also note that the key property of sunflowers formed by sets is as follows: if a $k$-element set intersects a sunflower with $k+1$ petals, then it must intersect the core. When $k$-sets are being replaced by $k$-dimensional vector spaces, we get the first, general position, definition, and not the second, disjointness,  definition.

\section{Preliminaries}

\subsection{Gaussian Coefficients}

Here we define \emph{Gaussian coefficients} (or \emph{$q$-binomial coefficients}).
For $q > 0$, write
\[
 [m]_q := \lim_{r \rightarrow q} \frac{r^m-1}{r-1} = \begin{cases}
           m & \text{ if } q = 1,\\
           \frac{q^m-1}{q-1} & \text{otherwise.}
          \end{cases}
\]
For $n \geq m \geq 0$, put
\[
 \gauss{n}{m}_q := \prod_{i=1}^m \frac{[n-i+1]_q}{[i]_q}.
\]
For $m > n$ or $m < 0$, put $\gauss{n}{m}_q = 0$.
Usually, $q$ is clear by context and we write $[m]$ for $[m]_q$ and $\gauss{n}{m}$ for $\gauss{n}{m}_q$.
For a vector space $V$, let $\gauss{V}{m}$ denote its family of $m$-subspaces.
For $q$ a prime power and $V = V(n, q)$, we have $\left| \gauss{V}{m} \right| = \gauss{n}{m}$.

\begin{Lemma}\label{lem:gauss_bnds}
  Let $m \geq 2$. Then
  \[(1 + \tfrac1q) q^{m-1} \leq [m] < (1 + \tfrac{1}{q-1}) q^{m-1}.\]
\end{Lemma}
\begin{proof}
  For the lower bound, we use $\frac{q^m-1}{q-1} = q^{m-1}+q^{m-2}+\ldots+1\ge q^{m-1}+q^{m-2} = (1+1/q)q^{m-1}$.
  For the upper bound we use $[m] = \frac{q^m-1}{q-1} < \frac{q^m}{q-1} = (1 + \frac{1}{q-1}) q^{m-1}$.
\end{proof}

\subsection{Subspace Codes}

We can define a metric on the family of subspaces of $V(n, q)$ with distance-function $\delta$ by putting
\[
 \delta(S, T) = \dim(S+T) - \dim(S \cap T) = \dim(S) + \dim(T) - 2\dim(S \cap T).
\]
A family $\cC$ of $m$-spaces of $V(n, q)$ such that $\delta(S, T) \geq D$ for all $S,T \in \cC$ is
called an \emph{$(n, D; m)$-subspace code}.
Based on Gabidulin codes \cite{Gabidulin1985}, Silva, Kschischang and K\"otter showed
the following \cite[\S{}IV.A]{SKK2007}.

\begin{Theorem}[Lifted MRD Codes]\label{thm:lifted_MRDorig}
 For any $n \geq m \geq 0$ and $D \leq 2\min(m, \allowbreak n-m)$ even,
 there exists an $(n, D; m)$-subspace code $\cC$ of size
 \[
  q^{\max(m, n-m) (\min(m, n-m) - D/2 + 1)}
 \]
 such that all elements of $\cC$ are disjoint to a fixed $(n-m)$-space.
\end{Theorem}

Note that for almost all valid parameters larger $(n, D; m)$-subspace
codes than those in Theorem \ref{thm:lifted_MRDorig}
are known, see also \cite{HK2017}.
The property that all elements of $\cC$ are disjoint
from a fixed $(n-m)$-space is crucial for us.

\smallskip

\noindent
For our applications, it is more convenient to rephrase
Theorem \ref{thm:lifted_MRDorig} as follows.

\begin{Proposition}\label{prop:lifted_MRD}
 For any $n \geq m \geq d \geq 0$,
 there exists a family of $m$-spaces $\cC$ in $V(n, q)$
 of size $q^{d(n-m)}$ if $n \geq 2m$, respectively,
 of size $q^{m(n-2m+d)}$ if $n \leq 2m$ and $n-2m+d \geq 0$, such that
 \begin{enumerate}
  \item any $d$-space lies in at most one element of $\cC$,
  \item all elements of $\cC$ are disjoint to a fixed $(n-m)$-space.
 \end{enumerate}
\end{Proposition}
\begin{proof}
  The first condition is equivalent to $\dim(C \cap C') \leq d-1$
  for any distinct $C, C' \in \cC$. As $\dim(C) = \dim(C') = m$,
  we may reformulate the first condition as $\delta(C, C') \geq 2(m-d+1)$.
  That is, $\cC$ is an $(n, D; m)$-code with minimum distance $D = 2(m-d+1)$.
  The case $n-2m+d = 0$ is trivial.
  Theorem \ref{thm:lifted_MRDorig} shows the assertion in all other cases.
\end{proof}

\subsection{The Upper Bound}

For a family of subspaces $\cS$,
we write $\sum_{S \in \cS} S$ for the span of all subspaces in $\cS$.
Let $V$ be an $n$-dimensional vector space over a field $\Ff$.
If $K$ is an $i$-space in $V$,
then \textit{quotient space} $V/K$ is isomorphic to the $(n-i)$-dimensional
vector space over $\Ff$. If $A$ is a $k$-space in $V$ with $K \subseteq A$,
then $A/K$ is a $(k-i)$-space in $V/K$.
More generally, if $\dim(A \cap K) = j$,
then $(A+K)/K$ is a $(k-j)$-space in $V/K$.

\begin{Lemma}\label{lem:upper_bnd}
 Let $s \geq 3$. Let $\cF$ be an $s$-sunflower-free family of $k$-spaces.
 Then $|\cF| \leq [k(s-1)] \cdot [(k-1)(s-1)] \cdots [s-1]$.
\end{Lemma}
\begin{proof}
 Note that for a fixed $i$-space $K$, $\cF(K) = \{ A/K: A \in \cF, K \subseteq A \}$
 is an $s$-sunflower-free family of $(k-i)$-spaces.

We assume that $\cF$ is maximal and show the claim by induction.
 For $k=1$, clearly $|\cF| \leq [s-1]$.
 For $k>1$, let us find a maximal collection $\cS$
 of pairwise disjoint $k$-spaces in $\cF$. Since $\cF$ is $s$-sunflower-free, we have $|\cS|\le s-1$.
 Put $M := \sum_{S \in \cS} S$. Then $\dim(M) \le (s-1)k$.
 By maximality of $\cS$, each element of $\cF$
 meets $M$ in an at least $1$-dimensional subspace. The subspace $M$ has
 at most $[k(s-1)]$ $1$-spaces, and each $1$-space $P$ lies in
 at most $[(k-1)(s-1)]\cdots [s-1]$ members of $\cF$ by induction.
\end{proof}

Now the quantitative upper bound in Theorem~\ref{thm:main} follows via an application of Lemma~\ref{lem:gauss_bnds}.

\section{Constructions for Sunflower-Free Families}

\subsection{Expository Examples}

As mentioned in the introduction, an example  for set systems that is expected to be close to extremal  consists
of the $(s-1)^k$ $k$-sets which meet each of $k$ fixed pairwise disjoint $(s-1)$-sets in one element.
A natural vector space analog is the family of all $k$-spaces which meet each of $k$ fixed $(s-1)$-spaces
in general position in a $1$-spaces. This gives an example for an $s$-sunflower-free family $\cF$ of size $[s-1]^k \le  q^{(s-1)k}$ (cf. Lemma~\ref{lem:gauss_bnds}).

We can construct a much larger example via an application of Proposition \ref{prop:lifted_MRD}. Consider the largest family $\cG$ of $k$-spaces in $V(\frac{(s+1)k}{2}-1, q)$
such that no $(\frac{k}{2}+1)$-space lies in more than one element of $\cG$. By Proposition \ref{prop:lifted_MRD}, for $s \geq 4$, it satisfies $|\cG| \ge q^{\frac{(s-1)k^2}4+\frac{(s-2)k}2 -1}$.
The family $\cG$ is $s$-sunflower-free: by construction, any sunflower in $\cG$
can only have a kernel of dimension at most $\frac{k}{2}$.
Thus, if $\cG$ contains an $s$-sunflower $\cS$, then $\dim(\sum_{S \in \cS} S) \geq \frac{k}{2} + \frac{ks}{2}$,
which is impossible in $V(\frac{(s+1)k}{2}-1, q)$. We see that $\cG$ is usually much larger than $\cF$; a situation which
does not occur in the set case. This is much closer to the upper bound, but the exponent is roughly a factor of $2$ off.

In general, we can do better.
Our goal is to describe two general constructions which capture the
idea of the following example.

\begin{Example}\label{ex:small}
Consider $s=3$ and $k=2$.
First we fix a $1$-space $T$ in $V = V(5, q)$.
We find $q^2+1$ $3$-spaces $\Pi_1, \ldots, \Pi_{q^2+1}$
through $T$ such that $\Pi_i \cap \Pi_j = T$.\footnote{%
In the quotient of $T$ we are in $V(4, q)$.
The $q^2+1$ $1$-spaces of $V(2, q^2)$ correspond to a partition of the non-zero vectors of $V(4, q)$
into $2$-spaces, a so-called \emph{spread}.}
In $\Pi_1$ we take all $q^2+q+1$ $2$-spaces.
In all other $\Pi_i$'s we take the $q^2$ $2$-spaces
disjoint from $T$. This yields an example $\cF$ of size
\[
 q^2 \cdot q^2 +q^2+ q+1 = q^4 + q^2+q+1.
\]
The upper bound in this case is
\[
  (q^3+q^2+q+1)(q+1) = q^4 + 2q^3 + 2 q^2 + 2 q + 1.
\]
The family is $3$-sunflower free:
Suppose that $\cS = \{ S_1, S_2, S_3 \}$ is a sunflower
 of the example with kernel $K$.
 If $\dim(K) = 0$, then $\dim(S_1 + S_2 + S_3) = 6$, which is
 impossible.
 Thus, $\dim(K) = 1$. By construction, for $L_1 \in \Pi_i$
 and $L_2 \in \Pi_j$ for $i \neq j$, we have $\dim(L_1 \cap L_2) = 0$.
 Thus, $S_1, S_2, S_3 \in \Pi_i$ for some $i$.
 But $\dim(\Pi_i) = 3$, while $\dim(S_1+S_2+S_3) = 1 + 3 \cdot 1 = 4$, which is, again, impossible.

The family $\cF$ is maximal: If we add a $2$-space disjoint
from $V$, then we find a $3$-sunflower with a kernel of dimension $0$.
If we add a $2$-space meeting $V$ in a $1$-space, then we find a $3$-sunflower with a kernel of dimension $1$.
If we add any of the $2$-spaces within $V$, then it is easy to see that
we find a $3$-sunflower with a kernel of dimension $1$.
\end{Example}

\subsection{A Proposition}

Let $\Sigma$ and $T$ be subspaces of an $n$-dimensional vector space $V$
with $\dim(\Sigma) = \sigma$, $\dim(T) = \tau$, and $T \subseteq \Sigma$.
Fix $d \leq m-\tau$. Let $\cC$ be a family of $m$-dimensional subspaces of $V$ such that
\begin{enumerate}
 \item $\Sigma \cap C = T$ for all $C \in \cC$,
 \item any $d$-space disjoint from $T$ lies in at most one element of $\cC$.
\end{enumerate}
Call such a family $\cC$ an \emph{$(m, d; V, \Sigma, T)$-family}.
The intuition is that, when taking the quotient over $T$, we obtain a subspace code that is disjoint from $\Sigma$.
In the constructions in the subsequent sections
we want to want to nest $(m, d; V, \Sigma, T)$-families within each other.

\begin{Proposition}\label{prop:MRDbnd2}
 Let $n, m, d,\tau, \sigma$ be nonnegative integers
 satisfying $m, n-m \geq d+\tau$ and $\sigma - \tau \leq n-m$.
 Put $V = V(n, q)$.
 Let $\Sigma \in \gauss{V}{\sigma}$ and $T \in \gauss{\Sigma}{\tau}$.
 Then there exists an $(m, d; V, \Sigma, T)$-family $\cC$ such that
 \begin{enumerate}
  \item if $n \geq 2m-\tau$, then $|\cC| \geq q^{d(n-m)}$;
  \item if $n \leq 2m-\tau$, then $|\cC| \geq q^{(m-\tau)(n-2m+\tau+d)}$.
 \end{enumerate}
\end{Proposition}
\begin{proof}
  We have $\dim(V/T) = n-\tau$,
  $\dim(S/T) = m-\tau$ for an $m$-space $S$ of $V$ with $T \subseteq S$, and
$\dim(\Sigma/T) = \sigma - \tau \leq n-m = (n-\tau) - (m-\tau)$.
  Note that $n-\tau\ge 2(m-\tau)$ corresponds to the first case,
  and $n - \tau \le 2(m-\tau)$ corresponds to the second case.
  We can apply Proposition \ref{prop:lifted_MRD}
  to $V/T$ and obtain the assertion.
\end{proof}

Proposition \ref{prop:MRDbnd2} is the basis for our two constructions.

\subsection{A Construction for \texorpdfstring{$s \geq k+1$}{s >= k+1}}

Let $s \geq k+1 \geq 3$. (For $k=1$, clearly $[s-1]$ is the extremal value.)
Then we can construct a family $\cA = \cA(s, k)$ in $V(ks-1, q)$ iteratively as follows:
For $i \in \{ 1, \ldots, k-1 \}$, put
\[
 (m_i, d_i) = (k-1+i(s-1), k-i).
\]
Fix subspaces $T_1 \subseteq \ldots \subseteq T_{k-1} \subseteq V = V(sk-1, q)$ with $\dim(T_i) = 2m_i-m_{i+1} = m_{i-1}$
(where we put $m_{k} = sk-1$).
For the following, we will apply Proposition \ref{prop:MRDbnd2}.
Let $\cC_{k-1}$ be a maximum size $(m_{k-1}, d_{k-1}; V, T_{k-1}, T_{k-1})$-family.

For each $C_i \in \cC_i$, $i \in \{ 2, \ldots, k-1\}$ replace $C_i$
with a maximum size $(m_{i-1}, d_{i-1}; C_i, T_{i}, T_{i-1})$-family $\cC_{i-1}$.
(Note that the family $\cC_{i-1}$ depends on $C_i$, and, as such, we generate a whole tree of such families, starting from the root at level $k$. But we suppress all indices other than the level in that tree for brevity.)
At the end, replace $C_1 \in \cC_1$ with a maximum size $(k, k; C_1, T_1, \{ 0 \})$-family $\cC_0$,
i.e., the family of all $k$-subspaces in $C_1$ (of dimension $k+s-2 \geq 2k-1$) avoiding $T_1$ (of dimension $k-1$).

This results in a family $\cA$ of $k$-spaces of size
\[
 |\cC_0| \cdot |\cC_1| \cdots |\cC_{k-1}|.
\]

Let us show that the family $\cA$ has the desired properties.
\begin{Lemma}\label{lem:wearedisjoint}
  Let $S, S' \in \cA$ with $S \subseteq C \in \cC_{i-1}$ and $S' \subseteq C' \in \cC_{i-1}$
  with $C \neq C'$ for some $(m_{i-1}, d_{i-1}; C_i, T_i, T_{i-1})$-family $\cC_{i-1}$.
  Then $\dim(S \cap S') \leq k-i$.
\end{Lemma}
\begin{proof}
 By construction, $\dim(S \cap T_{i-1}) = \dim(S' \cap T_{i-1}) = 0$
 and $(S \cap S') + T_{i-1} \subseteq C \cap C'$.
 As $C, C' \in \cC_{i-1}$, $\dim((C\cap C')/T_{i-1}) \leq d_{i-1}-1 = k-i$.
 Hence, $\dim(S \cap S') \leq k-i$.
\end{proof}

\begin{Lemma}\label{lem:nosunflower}
 The family $\cA$ does not contain an $s$-sunflower.
\end{Lemma}
\begin{proof}
 Suppose that $\cS$ is an $s$-sunflower with kernel $K$ in $\cA$.
Let $i$ be the smallest $i$ such that some $\cC_{i-1}$ contains $\cS$, i.e., that each petal belongs to the subfamily of subspaces generated from $\cC_{i-1}$. Then, for some $S,S'\in \cS$ we must have $S\subset C, S'\subset C'$ for distinct $C,C'\in \cC_{i-1}$. Therefore, we may apply  Lemma \ref{lem:wearedisjoint} and get that the kernel $K$ of $\cS$ satisfies $\dim(K) \le k-i$.
 But $\dim(S_1 + \cdots + S_s) \geq (k-i) + si = k+(s-1)i > k-1+i(s-1) = \dim(C_{i})$, a contradiction.
\end{proof}

\begin{Lemma}
 The family $\cA$ has size at least $q^{(s-1) \binom{k+1}{2} - k}$.
\end{Lemma}
\begin{proof}
 For $\cC_0$ we find $|\cC_0| \geq q^{k(s-2)}$ by Proposition \ref{prop:MRDbnd2}
 as $n=k+s-2$, $m=k$, $d=k=m$, $\sigma = k-1$ and $\tau = 0$.
 We also apply Proposition \ref{prop:MRDbnd2} for the remaining $\cC_i$, $i \in \{ 1, \ldots, k-1 \}$,
 with $n = m_{i+1}$, $m = m_i$, $d=k-i$, $\sigma = m_i$, and $\tau = m_{i-1}$.
 We find $|\cC_i| \geq q^{(k-i)(s-1)}$.
 Then $k(s-2) + \sum_{i=1}^{k-1} (k-i)(s-1) = (s-1) \binom{k+1}{2} - k$ shows the claim.
\end{proof}

This shows the first lower bound of Theorem \ref{thm:main}. We note that the family we constructed is not maximal.
It might still be true that all largest $s$-sunflower-free families of $k$-spaces
essentially contain a nesting of lifted MRD codes with the parameters
chosen as in our construction of $\cA$.

\subsection{A Construction for \texorpdfstring{$s \leq k$}{s <= k}}

Let $3 \leq s \leq k$.
In this case our construction is very similar to the construction of $\cA$,
but we are constrained by the conditions of Proposition \ref{prop:MRDbnd2}
which prevents us from nesting $k$ lifted MRD codes.
Here, we will only nest $s-1$ lifted MRD codes.

We construct a family $\cB = \cB(s, k)$ in $V(ks-1, q)$ iteratively as follows:
For $i \in \{ 1, \ldots, s-1 \}$, put
\[
 (m_i, d_i) = \left( ik-1, \left\lfloor \frac{s-1-i}{s-1} k \right\rfloor +1\right).
\]
Let $T_2 \subseteq \ldots \subseteq T_{s-1} \subseteq V = V(sk-1, q)$ with $\dim(T_i) = 2m_i-m_{i+1}=m_{i-1}$
(where we put $m_{s} = sk-1$).
Let $\cC_{s-1}$ be a maximal $(m_{s-1}, d_{s-1}; V, T_{s-1}, T_{s-1})$-family.
As when constructing $\cA$, we iteratively replace subspaces $C\in \cC_i$ by $\cC_{i-1}$'s.
At the end, replace $C_2 \in \cC_2$ with a maximal $(k, d_1; C_2, T_2, \{ 0 \})$-family $\cC_1$.
We end up with  a family $\cB$ of $k$-spaces of size
\[
 |\cC_1| \cdot |\cC_2| \cdots |\cC_{s-1}|.
\]
Note that
\[
 \left\lfloor \frac{s-1-i}{s-1} k \right\rfloor  + s\left(k - \left\lfloor \frac{s-1-i}{s-1} k \right\rfloor\right) > (i+1)k-1,
\]
thus, $\cC_i$ cannot contain an $s$-sunflower with kernel $\lfloor \frac{s-1-i}{s-1} k - 1 \rfloor$ or less.
Then, arguing as for $\cA$ in Lemma \ref{lem:wearedisjoint} and Lemma \ref{lem:nosunflower},
we conclude that $\cB$ does not contain an $s$-sunflower.

\begin{Lemma}
 The family $\cB$ has size at least
 \[q^{(s-2)\binom{k+1}{2} + k   \frac{\mathrm{gcd}(k, s-1) - 3}{2}}.\]
\end{Lemma}
\begin{proof}
 Apply Proposition \ref{prop:MRDbnd2}
 with $n=2k-1$, $m=k$, $d=\lfloor \frac{s-2}{s-1} k \rfloor + 1$,
 $\sigma = k-1$, $\tau = 0$. We find $|\cC_1| \geq q^{k \lfloor \frac{s-2}{s-1} k \rfloor }$.
 For $\cC_i$, $i \in \{ 2, \ldots, s-2 \}$, apply Proposition \ref{prop:MRDbnd2}
 with $n=(i+1)k-1$, $m=ik-1$, $d=\lfloor \frac{s-1-i}{s-1} k \rfloor + 1$,
 $\sigma = ik-1$, and $\tau = (i-1)k-1$.
 We find
 $|\cC_i| \geq q^{k(\lfloor \frac{s-1-i}{s-1} k \rfloor + 1)}$.

 For positive integers $\mu, \nu$ we have, see \cite[p. 90]{ConcreteMath},
 \[
  \sum_{\kappa=1}^{\nu-1} \left\lfloor \frac{\kappa \mu}{\nu} \right\rfloor
 = \frac{(\mu-1)(\nu-1) + \mathrm{gcd}(\mu, \nu) - 1}{2}.
 \]
 Hence, the sum of exponents equals
 \begin{align*}
   & k \left\lfloor \frac{s-2}{s-1} k \right\rfloor + \sum_{i=2}^{s-2} k\left(\left\lfloor \frac{s-1-i}{s-1} k \right\rfloor + 1\right)\\
   &= (s-3)k + k \sum_{i=1}^{s-2} \left\lfloor \frac{ik}{s-1} \right\rfloor\\
   &= (s-3)k + k \frac{(s-2)(k-1)+\mathrm{gcd}(k, s-1) - 1}{2}\\
   &= (s-2)\binom{k+1}{2} + k \frac{\mathrm{gcd}(k, s-1) - 3}{2}. \qedhere
 \end{align*}
\end{proof}

This shows the second lower bound of Theorem \ref{thm:main}.

We are not aware of any reasonable set of parameters for which a construction containing $\cB$ does
not seem to be the largest. In contrast to $\cA$, there are constructions which nest lifted MRD codes of different parameters
and whose size differs from $|\cB|$ only in a minor power of $q$.
It might be true that a largest sunflower-free family of $k$-spaces contains
a large nesting of lifted MRD codes as $\cB$ does, but, unlike for $\cA$,
it is not clear that our choice of parameters is always optimal.
The following conjecture is not unlikely.

\begin{Conjecture}
 Let $s \leq k$ and $q$ be a prime power. Then there exists a constant $C = C(q, s)$
 such that any $s$-sunflower-free family $\cF$ of $k$-spaces
 over the field with $q$ elements satisfies $|\cF| \leq |\cB(s, k)| \cdot C^k$.
\end{Conjecture}

\section{Open Problems}

We have already asked if it is true that any large sunflower-free family
of $k$-spaces necessarily contains a large nesting of lifted MRD codes.
As complementary questions, we have to ask if there exists some fundamentally
different construction for large sunflower-free families.

The behavior for $s \leq k$ of the Erd\H{o}s-Rado sunflower problem
for vector spaces is unclear. This is a dimension-less problem.
For answering this question, it might be worthwhile to
work out the largest $s$-sunflower-free families in $V(n, q)$ for fixed $n$.

For an $s$-set-sunflower-free family $\cF$ of $k$-spaces, we can treat $\cF$ as a family
of $[k]$-sets for which only $k$ intersection sizes $[i]$ for $i \in \{ 0, \ldots, k-1 \}$ are allowed.
For instance, see \cite{JJSW2025} for recent work on the
Erd\H{o}s-Rado sunflower problem with restricted intersection sizes.
Furthermore, this will connect to the work on subspace codes \cite{ER2015}.

\paragraph*{Acknowledgements}
The research of Andrey Kupavskii was supported by the Ministry of Economic Development of the Russian Federation (agreement with MIPT No. 139-15-2025-013, dated June 20, 2025, IGK 000000C313925P4B0002).

\end{document}